\documentclass{article}
\usepackage{amsmath,amssymb,theorem}
\usepackage{verbatim}

\usepackage{color}
\usepackage{makeidx}

\theorembodyfont{\upshape}

 \makeatletter
 \@addtoreset{equation}{}
 \makeatother

 \newtheorem{ittheorem}{Theorem}
 \newtheorem{itlemma}{Lemma}
 \newtheorem{itproposition}{Proposition}
 \newtheorem{itdefinition}{Definition}

 \newtheorem{itremark}{Remark}
 \newtheorem{itclaim}{Claim}
 \newtheorem{itcorollary}{\bf Corollary}

 \newenvironment{theorem}{\addtocounter{equation}{1}
 \begin{ittheorem}}{\end{ittheorem}}

 \newenvironment{lemma}{\addtocounter{equation}{1}
 \begin{itlemma}}{\end{itlemma}}

 \newenvironment{proposition}{\addtocounter{equation}{1}
 \begin{itproposition}}{\end{itproposition}}

 \newenvironment{definition}{\addtocounter{equation}{1}
 \begin{itdefinition}}{\end{itdefinition}}

 \newenvironment{remark}{\addtocounter{equation}{1}
 \begin{itremark}}{\end{itremark}}

 \newenvironment{claim}{\addtocounter{equation}{1}
 \begin{itclaim}}{\end{itclaim}}

 \newenvironment{proof}{\noindent {\bf Proof.\,}
 }{\hspace*{\fill}$\qed$\medskip}

 \newenvironment{corollary}{\addtocounter{equation}{1}
 \begin{itcorollary}}{\end{itcorollary}}

 \newcommand{\be}[1]{\begin{eqnarray*}\label{#1}}
 \newcommand{\ee}{\end{eqnarray*}}

 \newcommand{\bl}[1]{\begin{lemma}\label{#1}}
 \newcommand{\el}{\end{lemma}}

 \newcommand{\br}[1]{\begin{remark}\label{#1}}
 \newcommand{\er}{\end{remark}}

 \newcommand{\bt}[1]{\begin{theorem}\label{#1}}
 \newcommand{\et}{\end{theorem}}

 \newcommand{\bd}[1]{\begin{definition}\label{#1}}
 \newcommand{\ed}{\end{definition}}

 \newcommand{\bcl}[1]{\begin{claim}\label{#1}}
 \newcommand{\ecl}{\end{claim}}

 \newcommand{\bp}[1]{\begin{proposition}\label{#1}}
 \newcommand{\ep}{\end{proposition}}

 \newcommand{\bc}[1]{\begin{corollary}\label{#1}}
 \newcommand{\ec}{\end{corollary}}

 \newcommand{\bpr}{\begin{proof}}
 \newcommand{\epr}{\end{proof}}

 \newcommand{\bi}{\begin{itemize}}
 \newcommand{\ei}{\end{itemize}}

 \newcommand{\ben}{\begin{enumerate}}
 \newcommand{\een}{\end{enumerate}}

\def\un{{1\cdots 1}}

\def\um{\{\,1,\dots,m\,\}}
\def\un{\{\,1,\dots,N\,\}}

\def\uro{\smash{{U}^{\!\!\!\!\raise5pt\hbox{$\scriptstyle o$}}}}

\def\um{\{\,1,\dots,m\,\}}

\def\bp{{\overline{p}}}

\def\bp{{\overline{p}}}

 \def \ba {\begin{array}}
 \def \ea {\end{array}}

 \def \qed {{\heartsuit\hfill}}
 
 \def \R {{\mathbb R}}
 
 \def \N {{\mathbb N}}

\def \qed {{\square\hfill}}

\def \qed {{\square\hfill}}

\def\R{{\mathbb R}}

\def\N{{\mathbb N}}

\def\eqref#1{(\ref{#1})}

\usepackage{titlesec}
\newcommand{\periodafter}[1]{#1.}
\titleformat{\section}[runin]
{\normalfont\bfseries}{\thesection .}{4 pt}{\periodafter}

\makeindex

\begin{document}

\title{Galton--Watson and branching process representations of the 
normalized Perron--Frobenius eigenvector}

\author{
\qquad Rapha\"el Cerf 
\footnote{
\noindent
D\'epartement de math\'ematiques et applications, Ecole Normale Sup\'erieure,
CNRS, PSL Research University, 75005 Paris.
\newline
Laboratoire de Math\'ematiques d'Orsay, Universit\'e Paris-Sud, CNRS, Universit\'e
Paris--Saclay, 91405 Orsay.}
\hskip 70pt Joseba Dalmau
\footnote{
Centre de Math\'ematiques appliqu\'ees, Ecole Polytechnique CNRS, 
Universit\'e Paris--Saclay, 91405 Orsay.}
}

\maketitle



\begin{abstract}
\noindent
Let $A$ be a primitive matrix and let $\lambda$ be its Perron--Frobenius eigenvalue.
We give formulas expressing the associated normalized Perron--Frobenius eigenvector
as a simple functional of a multitype Galton--Watson process whose mean matrix is $A$,
as well as of a multitype branching process with 
mean matrix $e^{(A-I)t}$.
These formulas are generalizations of the classical formula
for the invariant probability measure of a Markov chain.
\noindent
\end{abstract}


\allowdisplaybreaks[4]

\noindent
Let $A$ be a primitive matrix of size $N$,
i.e., a non--negative matrix whose $m$--th 
power is positive for some natural number $m$.
The Perron--Frobenius theorem (Theorem 1.1 in~\cite{SEN})
states that there 
exist a positive real number $\lambda$ and a 
vector $u$ with positive coordinates
such that
$u^TA=\lambda u^T$.
Moreover, the eigenvalue $\lambda$ is simple,
is larger in absolute value than any other eigenvalue of $A$,
and any non--negative eigenvector of $A$ is a multiple of $u$.
The eigenvalue $\lambda$ is the Perron--Frobenius eigenvalue of $A$ and $u$
is a Perron--Frobenius eigenvector of $A$.
The purpose of this note is to give probabilistic representations
of the normalized Perron--Frobenius eigenvector $u/|u|_1$, as a functional of a
multitype Galton--Watson process,
as well as of a multitype branching process.

\section{The Galton--Watson case}
A multitype Galton--Watson process is a Markov chain
$$Z_n\,=\,\big(Z_n(1),\dots,Z_n(N)\big)\,,\qquad n\geq 0\,,$$
with state space $\N^N$.
The number $Z_n(i)$ represents the number of individuals 
having type $i$ in generation $n$.
In order to build generation $n+1$ from generation $n$,
each individual of type $i$ present in generation $n$ produces a random number of offspring,
distributed according to a prescribed reproduction law, independently of the other individuals
and the past of the process.
The ensemble of all the offspring forms the generation $n+1$.
The null vector is an absorbing state.
For each $i\in\lbrace\,1,\dots,m\,\rbrace$,
we denote by $P_i$ and $E_i$ the probabilities and expectations
for the process started from a population consisting of a single individual of type $i$.
From now onwards, 
we consider 
a multitype Galton--Watson model 
$(Z_n)_{n\geq0}$
whose
mean matrix is equal to $A$, i.e., we suppose that
$$\forall\, i,j\in
\{\,1,\dots,N\,\}
\qquad
E_i\big( Z_1(j) \big)\,=\,
A(i,j)\,.$$
There exist 
intimate links between the asymptotic behavior of the 
Galton--Watson process
$(Z_n)_{n\geq0}$, the Perron--Frobenius eigenvalue $\lambda$, and the associated normalized
eigenvector $u$ of $A$.
For instance, the following classical result can be found in
Chapter 2 of~\cite{Harris}.
If the Perron--Frobenius eigenvalue $\lambda$
of $A$ is strictly larger than one, then the multitype Galton--Watson process
has a positive probability of survival.
Conditionally on the survival event,
the vector of proportions of the different types 
converges almost surely to $u$ when time goes to $\infty$,
i.e., conditionally on the survival event, with probability one,
$$\forall\,i\in\lbrace\,1,\dots,N\,\rbrace\qquad
\lim_{n\to \infty}\,\frac{Z_n(i)}{Z_n(1)+\cdots+Z_n(N)}\,=\,u(i)\,.$$
In this note, we shall present a simpler formula linking the 
Galton--Watson process
$(Z_n)_{n\geq0}$ with $\lambda$ and $u$,
which works not only in the case $\lambda>1$,
but in the critical and subcritical cases also.
Let us fix $i\in 
\{\,1,\dots,N\,\} $.
We shall stop the process
$(Z_n)_{n\geq0}$ on the type $i$ by killing the descendants of individuals of type~$i$ in
any generation $n\geq 2$.
The resulting process is denoted by
$(Z^i_n)_{n\geq0}$.
Thus, in the stopped process
$(Z^i_n)_{n\geq0}$, the individuals reproduce as in the 
Galton--Watson process
$(Z_n)_{n\geq0}$, however
from generation~$1$ onwards, 
the individuals of type $i$ 
do not have any descendants.
%
We denote by $E_i$ the expectation for the process 
$(Z^i_n)_{n\geq0}$
starting with a population consisting
of one individual of type~$i$. Notice that this individual produces offspring as in 
the Galton--Watson process
$(Z_n)_{n\geq0}$, only individuals of type~$i$ belonging to the subsequent generations
are prevented from having offspring.
Finally, for $u=(u(1),\dots,u(N))$ a vector in $\R^N$, we define 
$$|u|_1 \,=\,|u(1)|+\cdots+|u(N)|\,.$$
%
%
\begin{theorem}
\label{mainth}
The normalized Perron--Frobenius eigenvector $u$ of $A$ is given by the formula
$$\forall i\in\un\qquad
	u(i)\,=\,
\frac{1}{
\displaystyle \sum_{n\geq 1} 
	\lambda^{-n} 
E_i \big( \big|Z_n^i\big|_1\big) 
}\,.
$$
\end{theorem}
Notice that $\big|Z^i_n|_1$ 
is simply 
the size of $n$--th generation of the process
$(Z^i_n)_{n\geq0}$.
In the case where $\lambda\geq 1$, 
the factor $\lambda^{-n}$ is naturally
interpreted as a killing probability.
We introduce a random clock $\tau_\lambda$,
independent of the branching process
$(Z_n)_{n\geq0}$, and distributed according to the geometric law of parameter $1-1/\lambda$:
$$\forall n\geq 1\qquad P(\tau_\lambda\geq n)\,=\,\Big(\frac{1}{\lambda}\Big)^{n-1}\,.$$
The formula presented in the theorem can then be rewritten as
\begin{equation*}
\label{pfrepp4}
\forall i\in E\qquad
	u(i)\,=\,\frac{1}{
	\displaystyle
E_i\Bigg(\sum_{n=1}^{\tau_\lambda-1}
 \big|Z_n^i\big|_1
\Bigg)
}
\,.
\end{equation*}
The nicest situation is when the Perron--Frobenius eigenvalue is equal to one. In this case, 
the formula becomes
$$\forall i\in\un\qquad
	u(i)\,=\,
\frac{1}{
\displaystyle 
E_i \Big( 
\sum_{n\geq 1} 
\big|Z_n^i\big|_1\Big) 
}\,.
$$
The denominator is naturally interpreted as the expected number of descendants from an individual of
type $i$, if the descendants of type $i$ are forbidden to reproduce.
Let us remark also that, by multiplying the matrix $W$ by a constant factor, we can adjust the 
value of the Perron--Frobenius eigenvalue without altering the Perron--Frobenius eigenvector.
More precisely, suppose that the mean matrix of the Galton--Watson process is given by
$$\forall\, i,j\in
\{\,1,\dots,N\,\}
\qquad
E_i\big( Z_1(j) \big)\,=\,
c\,A(i,j)\,,$$
where $c$ is a positive constant.
If we take $c=1/\lambda$, 
then we obtain indeed a critical branching process and the Perron--Frobenius eigenvalue is~$1$.
In practice, the exact value of the Perron--Frobenius eigenvalue might be unknown, so we can
simply choose a value $c$ large enough so that the 
Perron--Frobenius eigenvalue becomes larger than one, and we can introduce the random killing clock
as above.

In the particular case where the matrix $A$ is a stochastic matrix, and each individual produces exactly
one child, 
the Perron--Frobenius
eigenvalue $\lambda$ is equal to $1$ and the 
process 
$(Z_n)_{n\geq0}$ is simply a Markov chain with transition matrix $A$.
The stopped process
$(Z^i_n)_{n\geq0}$ is the Markov chain stopped at the time $\tau_i$ of the
first return to~$i$. So, in this situation, the population $Z^i_n$ has size $1$ until time $\tau_i$
and $0$ afterwards, therefore
	$$
\displaystyle \sum_{n\geq 1} 
\lambda^{-n} 
E_i \big( \big|Z_n^i\big|_1\big) 
	\,=\,
E_i\Bigg(\sum_{n\geq 1}^{\tau_i}
1
\Bigg)
\,=\,
E_i(\tau_i)
$$
and we recover the classical formula for the invariant probability measure of a Markov chain.
%

Let us come to the proof of the theorem.
The theorem is in fact a consequence of the following proposition.
\begin{proposition}
	\label{mainpr}
Let $i\in 
\{\,1,\dots,N\,\} $.
The vector $v$ defined by
$$
\forall j\in \{\,1,\dots,N\,\} 
	\qquad
	v(j)\,=\,
\displaystyle \sum_{n\geq 1} 
\lambda^{-n} 
E_i \big( Z_n^i(j) \big) 
$$
	is the Perron--Frobenius eigenvector of $A$ satisfying $v(i)=1$.
\end{proposition}

Indeed, the formula appearing in the theorem is obtained by normalizing the above vector.
We now proceed to the proof of the proposition.
\smallskip

\begin{proof}
We fix $i\in \{\,1,\dots,N\,\} $ and we define a vector $v$ via the formula stated in
the proposition.
	Let us examine first $v(i)$. By definition of the 
stopped process
$(Z^i_n)_{n\geq0}$, we have
	$$ v(i) \,=\,
\sum_{n\geq 1}
\sum_{i_1,\dots,i_{n-1}\neq i}
\lambda^{-n}A(i,i_1)\cdots A(i_{n-1},i)\,.
$$
The Lemma of \cite{CD} yields that
	the above sum is equal to~$1$, thus $v(i)=1$.
Let next
$j$ belong to $\{\,1,\dots,N\,\}$.
We have
	$$
	(vA)(j)\,=\,
	\sum_{1\leq k\leq N} v(k) A(k,j)
	\,=\, A(i,j)+
	\sum_{\genfrac{}{}{0pt}{1}{ 1\leq k\leq N }{ k\neq  i }}
	v(k) A(k,j)\,.
	$$
We compute next
\begin{align*}
	\sum_{\genfrac{}{}{0pt}{1}{ 1\leq k\leq N }{ k\neq  i }}
	v(k) A(k,j)
	&\,=\, 
	\sum_{\genfrac{}{}{0pt}{1}{ 1\leq k\leq N }{ k\neq  i }}
\sum_{n\geq 1} 
\lambda^{-n} 
E_i \big( Z_n^i(k) \big) A(k,j)
	\cr
	&\,=\, 
\sum_{n\geq 1} 
\lambda^{-n} 
E_i \Big( 
	\sum_{\genfrac{}{}{0pt}{1}{ 1\leq k\leq N }{ k\neq  i }}
	Z_n^i(k) 
	A(k,j)
	\Big) 
	\cr
	&\,=\, 
\sum_{n\geq 1} 
\lambda^{-n} 
E_i \bigg( 
	E\Big(Z_{n+1}^i(j)\,\Big|\,Z_n^i\Big)
	\bigg) 
	\cr
	&\,=\, 
\sum_{n\geq 1} 
\lambda^{-n} 
	E_i\big(Z_{n+1}^i(j)\big)
	\cr
	&\,=\, \lambda v(j) -
	E_i\big(Z_{1}^i(j)\big)\,.
\end{align*}
Remember that the initial individual of type~$i$ reproduces as in the Galton--Watson process
$(Z_n)_{n\geq0}$, therefore
	$E_i\big(Z_{1}^i(j)\big)=A(i,j)$ and putting together the previous computations, we obtain
	$$\forall j\in \{\,1,\dots,N\,\} 
	\qquad
	(vA)(j)
	\,=\, 
	\lambda v(j) 
	\,.$$
	Since in addition $v(i)=1$, we conclude that 
	all the components of $v$ are positive and finite, 
	therefore $v$ is a left Perron--Frobenius eigenvector of $A$,
	as wanted.
\end{proof}

\section{The branching process case}
A multitype branching process is a 
continuous--time
Markov process 
$$Z_t\,=\,\big(
Z_t(1),\dots,Z_t(N)
\big)\,,\quad t\geq 0\,,$$
with state space $\N^N$.
The number $Z_t(i)$ represents 
the number of individuals carrying the type $i$
at time $t$. 
Individuals reproduce independently of each other,
at a rate dependent on their type. When an individual reproduces,
it gives birth to a random number of offspring, distributed
according to a prescribed reproduction law, independently of 
the other individuals and the past of the process.
The null vector is an absorbing state.
For each $i\in\um$ , we denote by $P_i$ and $E_i$
the probabilities and and expectations 
for the process started from a population
consisting of a single individual of type $i$.
We consider a multitype branching process 
whose mean matrix has generator $A-I$,
in other words, we suppose that
$$\forall\,i,j\in\un\quad \forall t\geq 0\qquad
E_i\big(
Z_t(j)
\big)\,=\,\big(e^{(A-I)t}\big)(i,j)\,,$$
where the exponential appearing in the formula
is the matrix exponential.
This mean matrix corresponds to the process
where individuals reproduce at rate 1,
and the reproduction laws are the same as in
the discrete Galton--Watson process.
There exist well--known links between the asymptotic
behavior of the branching process $(Z_t)_{t\geq0}$,
the Perron--Frobenius eigenvalue $\lambda$,
and the associated eigenvector $u$ of $A$,
analogous to those of the Galton--Watson case.
We shall next present a simple formula 
in the spirit of theorem~\ref{mainth}.
Let us fix $i\in\un$.
We will stop the process $(Z_t)_{t\geq0}$
on the type $i$ by killing the descendants of type $i$
at any time $t\geq0$.
The resulting process is denoted by $(Z_t^i)_{t\geq0}$.
We denote by $E_i$ the expectation for the process $(Z_t^i)_{t\geq0}$
starting from a random population,
drawn according to the original reproduction law
of an individual of type $i$, so that
$E_i(Z^i_0(j))=A(i,j)$.
All individuals of type $i$ die at rate 1
without producing offspring.
\begin{theorem}
The normalized Perron--Frobenius eigenvector
$u$ of $A$ is given by the formula
$$\forall\,i\in\un\qquad
u(i)\,=\,\frac{1}{\displaystyle
\int_0^\infty e^{-(\lambda-1) t}E_i\big(
\big|Z_t^i\big|_1
\big)\,dt}$$
\end{theorem}
As for the Galton--Watson case,
this result is a direct consequence of the 
following proposition.
\begin{proposition}
	\label{brmain}
Let $i\in 
\{\,1,\dots,N\,\} $.
The vector $v$ defined by
$$
\forall j\in \{\,1,\dots,N\,\} 
	\qquad
	v(j)\,=\,
\int_0^\infty e^{-(\lambda-1) t}E_i\big(
Z^i_t(j)
\big)\,dt
$$
is the Perron--Frobenius eigenvector of $A$ satisfying $v(i)=1$.
\end{proposition}
\begin{proof}
As in the discrete case,
we do the proof by verifying that the vector $v$
is indeed an eigenvector of $A$.
Let $k\in\un$,
and let us start by computing the integral
involved in the definition of $v(k)$.
Differentiating the expectation with 
respect to $t$ yields
$$\frac{d}{dt}E_i\big(
Z^i_t(k)
\big)\,=\,\sum_{\genfrac{}{}{0pt}{1}{ 1\leq j\leq N }{ j\neq  i }}
E_i\big(
Z^i_t(j)
\big)
A(j,k)-
E_i\big(
Z^i_t(k)
\big)\,.
$$
Thus, integrating by parts,
for any $T>0$,
\begin{multline*}
\!\!\!\!\!\int_0^T e^{-(\lambda-1) t}E_i\big(
Z^i_t(k)
\big)\,dt
\,=\,
-\frac{1}{\lambda-1}e^{-(\lambda-1) T}E_i\big(
Z^i_T(k)
\big)+\frac{1}{\lambda-1}E_i\big(
Z^i_0(k)
\big)
\\+
\frac{1}{\lambda-1}\int_0^T e^{-(\lambda-1) t}\bigg(
\sum_{\genfrac{}{}{0pt}{1}{ 1\leq j\leq N }{ j\neq  i }}
E_i\big(
Z^i_t(j)\big)A(j,k)-
E_i\big(
Z^i_t(k)\big)\bigg)\,dt\,.
\end{multline*}
Let $B$ be the matrix obtained from $A$
by filling with zeros its $i$--th row.
The first expectation on the right hand--side
can be rewritten as
$$E_i\big(
Z^i_t(k)
\big)\,=\,\sum_{1\leq j\leq N}A(i,j)\big(
e^{(B-I)t}
\big)(j,k)\,.$$
Yet it follows from part (e) of theorem 1.1
of~\cite{SEN} that the spectral radius of $B$
is strictly less than $\lambda$. Therefore, when $t$ goes to infinity,
the matrix exponential $e^{(B-I)t}$
behaves as $e^{\mu't}$, for some $\mu'$ strictly smaller than $\lambda-1$.
Sending $T$ to infinity in the above integrals we obtain the 
following identity:
$$\lambda v(k)\,=\,A(i,k)+
\sum_{\genfrac{}{}{0pt}{1}{ 1\leq j\leq N }{ j\neq  i }}
v(j)A(j,k)\,.
$$
Thus, the proof will be achieved if we manage to show that 
$v(i)=1$.
Yet, the previous formula holds for $k=i$ too,
and we may use it recurrently over $v(j)$
in order to get, for any $n\geq 1$,
\begin{multline*}
v(i)\,=\,
\sum_{t=1}^{n} \frac{1}{\lambda^n}
\sum_{i_1,\dots,i_t\neq  i}
A(i,i_1)\cdots A(i_t,i)
\\+\frac{1}{\lambda^n}
\sum_{i_1,\dots,i_{n+1}\neq  i}
v(i_1)A(i_1,i_2)\cdots A(i_{n+1},i)\,.
\end{multline*}
Again, calling $B$ the matrix obtained from $A$
by filling with zeros its $i$--th row, 
the last term can be written as
$$\frac{1}{\lambda^n}
\sum_{\genfrac{}{}{0pt}{1}{ 1\leq i_1\leq N }{ i_1\neq  i }}
v(i_1)B^n(i_1,i),$$
which converges to 0 when $n$ goes to $\infty$.
Thus,
$$ v(i) \,=\,
\sum_{n\geq 1}
\sum_{i_1,\dots,i_{n-1}\neq i}
\lambda^{-n}A(i,i_1)\cdots A(i_{n-1},i)\,.
$$
This last quantity is equal to 1, as shown in the lemma of~\cite{CD}.
\end{proof}

\bibliographystyle{amsplain}
\bibliography{gwre}

\providecommand{\bysame}{\leavevmode\hbox to3em{\hrulefill}\thinspace}
\providecommand{\MR}{\relax\ifhmode\unskip\space\fi MR }
\providecommand{\MRhref}[2]{%
  \href{http://www.ams.org/mathscinet-getitem?mr=#1}{#2}
}
\providecommand{\href}[2]{#2}
\begin{thebibliography}{1}

\bibitem{CD}
Rapha\"el Cerf and Joseba Dalmau, \emph{A {M}arkov chain representation of the
  normalized {P}erron-{F}robenius eigenvector}, Electron. Commun. Probab.
  \textbf{22} (2017), Paper No. 52, 6. \MR{3718702}

\bibitem{Harris}
Theodore~E. Harris, \emph{The theory of branching processes}, Springer-Verlag,
  Berlin; Prentice-Hall, Inc., Englewood Cliffs, N.J., 1963. \MR{0163361}

\bibitem{SEN}
E.~Seneta, \emph{Nonnegative matrices and {M}arkov chains}, second ed.,
  Springer Series in Statistics, Springer-Verlag, New York, 1981. \MR{719544}

\end{thebibliography}
 \thispagestyle{empty}

\end{document}